\documentclass[12pt]{amsproc}

\usepackage[T1]{fontenc}
\usepackage{lmodern}
\usepackage{microtype}
\usepackage{amsmath,amssymb,amsthm,mathtools}
\usepackage[backref=page,colorlinks=true,linkcolor=blue,citecolor=blue,urlcolor=blue]{hyperref}
\usepackage[a4paper,margin=32mm]{geometry}

\numberwithin{equation}{section}

\newtheorem{theorem}{Theorem}[section]
\newtheorem{proposition}[theorem]{Proposition}
\newtheorem{lemma}[theorem]{Lemma}

\newtheorem*{centralizerproblem}{The centralizer problem}
\theoremstyle{definition}
\newtheorem{definition}[theorem]{Definition}
\theoremstyle{remark}
\newtheorem{remark}[theorem]{Remark}

\newcommand{\M}{\mathcal M}
\newcommand{\U}{\mathcal U}
\newcommand{\tr}{\operatorname{tr}}
\newcommand{\Tr}{\operatorname{Tr}}
\newcommand{\dist}{\operatorname{dist}}

\newcommand{\EL}{\operatorname{EL}}

\newcommand{\id}{\mathrm{id}}

\title[A conditional construction of a nonhyperlinear group]
{A conditional construction of a nonhyperlinear group
and the centralizer problem}

\author{Andreas Thom}
\address{A.T., Institut f\"ur Geometrie, TU Dresden,
01062 Dresden, Germany}
\email{andreas.thom@tu-dresden.de}

\hypersetup{
  pdftitle={A conditional construction of a nonhyperlinear group and the centralizer problem},
  pdfauthor={Andreas Thom},
  pdfsubject={A reduction of the existence of a nonhyperlinear group to the centralizer problem},
  pdfkeywords={hyperlinear group, property (T), tracial ultraproduct, relative commutant, infranormal subgroup}
}


\begin{document}

\begin{abstract}
We show that a positive answer to the centralizer problem for tracial
matrix ultraproducts implies the existence of a finitely generated
nonhyperlinear group.
\end{abstract}

\maketitle

\tableofcontents

\section{Introduction}

In his 1976 classification paper,
Connes suggested that every separable $\mathrm{II}_1$ factor should embed
into an ultrapower $R^\omega$ of the hyperfinite $\mathrm{II}_1$ factor
\cite{Connes76}.  Kirchberg subsequently placed this question in a broad
circle of equivalences involving tensor products and the QWEP problem for
$C^*$-algebras
\cite{Kirchberg93}.  Work relating the embedding problem to Tsirelson's
problem connected it with finite-dimensional quantum correlations
\cite{JungeEtAl11}.  The theorem $\mathrm{MIP}^*=\mathrm{RE}$ of Ji,
Natarajan, Vidick, Wright, and Yuen ultimately disproved the Connes
embedding problem \cite{MIPRE}.

The corresponding group-theoretic problem remains open, while a recent breakthrough by OpenAI has produced the first example of a nonsofic group \cite{OAI}, see also \cite{KT}.  R\u{a}dulescu
introduced hyperlinear groups \cite{Radulescu08}: a countable group is
hyperlinear if it embeds into the unitary group of a tracial ultraproduct
of matrix algebras.  Elek and Szab\'o proved that every sofic group is
hyperlinear \cite{ElekSzabo05}.  Although the general Connes embedding
problem has a negative answer, no nonhyperlinear group is known.  

The study of sofic groups is in a sense more approachable, since it allows for input from finite combinatorics.  Kun
proved that every sofic approximation of a Kazhdan group is, up to a negligible
modification, a disjoint union of expanders \cite{Kun16}.  Building on
this decomposition, Kun and the author showed that the sufficiently
accurate almost automorphisms of such an expander sofic approximation
form a group in a natural way \cite{KunThom19}.  Alekseev and the author
subsequently proved that, after passing to an essentially equivalent
finite model, the centralizer of a sofic embedding of a Kazhdan group in
a universal sofic group is itself an internal metric ultraproduct of finite permutation
groups \cite{ATC}.  Motivated by these results, they formulated the
following matrix analogue \cite[Open Problem~6.2(a)]{ATC}, which we call
the centralizer problem.  

\begin{centralizerproblem}\leavevmode\par\noindent
Let $G$ be a Kazhdan group and let
$\pi\colon G\to U(\prod_{n\to\U}M_{d_n}(\mathbb C))$ be a
homomorphism.  Do there exist finite-dimensional $*$-subalgebras
$A_n\subseteq M_{d_n}(\mathbb C)$ such that
\[
 \pi(G)'\cap\prod_{n\to\U}M_{d_n}(\mathbb C)
 =\prod_{n\to\U}A_n?
\]
\end{centralizerproblem}

\begin{remark}
The formulation in \cite[Open Problem~6.2(a)]{ATC} allows first replacing
$d_n$ by $m_n$, where $m_n/d_n\to_{\U}1$.  This flexibility is in fact
unnecessary for the internality assertion above, by
Lemma~\ref{lem:dimension-invariance} below.  This contrasts with the situation for sofic approximations and with
the correction of approximate representations in normalized Hilbert--Schmidt
norm, for which a change of dimension can be essential; see Becker and
Lubotzky \cite[Appendix~A]{BL}.
\end{remark}

The purpose of this paper is to show that a
positive answer to this problem implies the existence of a finitely generated
nonhyperlinear group. The construction announced by OpenAI \cite{OAI}
inspired Kun and the author to isolate the mechanism used in the proof of \cite[Proposition 2.3]{OAI} in \cite{KT}.
For a subgroup $H<G$, put $P_H=\{t\in G:tHt^{-1}\leq H\}$, the
so-called \emph{compression semigroup} of $H$ in $G$.
Following \cite{KT}, we call $H$ \emph{infranormal} in $G$ if $P_H$
generates $G$.  
They proved that if $H<G$ are Kazhdan, $H$ is infranormal and
nonnormal, then both $G*_H(H\times\mathbb Z/2\mathbb Z)$ and the group
double $G*_H G$ are nonsofic \cite{KT}.  The key representation-theoretic
statement is that the
centralizer of the image of $H$ in a universal sofic group is normalized
by the image of $G$; see \cite[Theorem~4.1]{KT}.

We prove a matrix counterpart to this result, conditional on a positive answer
to the centralizer problem.

\begin{theorem}
\label{thm:normalization-intro}
Assume that the centralizer problem has a positive answer.  Let $H<G$ be
Kazhdan groups and suppose that $H$ is infranormal in $G$.  If
$\pi\colon G\to U(\M)$ is a homomorphism into a tracial matrix
ultraproduct, then
\[
 \pi(g)(\pi(H)'\cap\M)\pi(g)^{-1}=\pi(H)'\cap\M
\]
for every $g\in G$.
\end{theorem}

The proof proceeds in two steps: we first prepare exact inclusions at the
finite-dimensional level, and then establish a phenomenon which we call
\emph{no-drift}.  Put $A=\pi(H)'\cap\M$ and $D=\pi(G)'\cap\M$.
Choose $t_1,\ldots,t_m\in P_H$ such that
$G=\langle H,t_1,\ldots,t_m\rangle$, and set $u_\ell=\pi(t_\ell)$.
Since $t_\ell Ht_\ell^{-1}\leq H$, one immediately has
$u_\ell^*Au_\ell\subseteq A$.  The centralizer problem gives aligned
internal models $D_n\subseteq A_n$, with representatives
$u_{\ell,n}\in D_n'$.  At finite level, the preceding inclusion becomes
the one-sided near inclusion
$A_n^-=u_{\ell,n}^*A_nu_{\ell,n}\subset_{2,\varepsilon_n}A_n^+=A_n$.
The relative correction theorem replaces it by an exact inclusion
$\widehat A_n^-\subseteq\widehat A_n^+$, with compatible identifications
of $D_n$ on asymptotically full corners.

Assign to a block $M_{r_i}(\mathbb C)\otimes1_{s_i}$ the scale $r_i/s_i$.
For an exact inclusion these scales are ordered, and a conditional-expectation
estimate bounds the reverse near-inclusion error in terms of the scale
quotient.  A $D_n$-conditional median gives a bounded transform with
conditional expectation $1/2$, and the common copy of $D_n$ allows the same
median on the source and target sides.  The scale order passes to the
ultraproduct, where equality of traces makes this transform invariant under
every $u_\ell$.  It therefore belongs to $D$ and equals $1/2$.  The resulting
concentration makes the conditional-expectation bound tend to zero, giving
the reverse near inclusion and hence $u_\ell^*Au_\ell=A$.

\medskip

The application is immediate once Theorem~\ref{thm:normalization-intro} is
available.

\begin{theorem}
\label{thm:main-intro}
Assume that the centralizer problem has a positive answer.  If $H < G$ is infranormal and not normal, and both groups are Kazhdan, then $G \ast_H G$ is not hyperlinear.

Moreover, let $q$ be
a prime power and let $r,d\geq3$.  Put
$R_+=\mathbb F_q[x_1,\ldots,x_d]$ and
$R=\mathbb F_q[x_1^{\pm1},\ldots,x_d^{\pm1}]$, and set
$H=\EL_r(R_+)$ and
$G=\EL_r(R)\rtimes\operatorname{SL}_d(\mathbb Z)$,
where $\operatorname{SL}_d(\mathbb Z)$ acts by monomial substitutions.  Then $H<G$ satisfies these assumptions.
\end{theorem}

The pair in Theorem~\ref{thm:main-intro} is the explicit residually finite
Kazhdan pair constructed in \cite[Theorem~E]{KT}.

This paper is organized as follows.  Section~\ref{sec:internal-algebras}
collects the necessary facts about inclusions of internal finite-dimensional
algebras.  Section~\ref{sec:correction} proves a relative correction theorem
for one-sided near inclusions.  Section~\ref{sec:no-drift} proves the
no-drift theorem and explains why the common internal subalgebra is
essential.  Section~\ref{sec:main-results} applies the theorem to
normalization of Kazhdan centralizers and then to the conditional
nonhyperlinearity of a group double.

\section{Inclusions of internal algebras}
\label{sec:internal-algebras}

Fix a nonprincipal ultrafilter $\U$ on $\mathbb N$.  All matrix algebras
carry their normalized traces.  We write
$\M=\prod_{n\to\U}(M_{d_n}(\mathbb C),\tr_n)$ for the tracial
ultraproduct and $\|x\|_2=\tr_n(x^*x)^{1/2}$ at the
finite level.  If $A_n\subseteq M_{d_n}(\mathbb C)$ are unital
$*$-subalgebras, then $\prod_{n\to\U}A_n$ denotes their tracial
ultraproduct inside $\M$.  These are the usual metric ultraproducts of
finite tracial von Neumann algebras; for general background on von Neumann
algebra ultraproducts, see \cite{AndoHaagerup}.

\begin{definition}
A countable group $\Lambda$ is \emph{hyperlinear} if there are a
nonprincipal ultrafilter $\U$, integers $d_n$, and an injective
homomorphism $\Lambda\to U(\prod_{n\to\U}M_{d_n}(\mathbb C))$.
\end{definition}

For finite-dimensional subalgebras $B,A\subseteq M_d(\mathbb C)$, put
\[
 B\subset_{2,\varepsilon}A
 \quad\Longleftrightarrow\quad
 \sup_{b\in B,\ \|b\|\leq1}\dist_2(b,A)\leq\varepsilon.
\]
Let $E_A$ and $E_B$ be the trace-preserving conditional expectations,
and put $\|T\|_{\infty,2}=\sup_{\|x\|\leq1}\|T(x)\|_2$ for a linear
map $T$.  Then
\begin{equation}
 B\subset_{2,\varepsilon}A
 \quad\Longleftrightarrow\quad
 \|(1-E_A)E_B\|_{\infty,2}\leq\varepsilon.
 \label{eq:near-inclusion-expectations}
\end{equation}
Indeed, $E_A$ is the orthogonal projection onto $A$ in $L^2$, and the
image under $E_B$ of the ambient operator-norm unit ball is precisely the
operator-norm unit ball of $B$.  This is the classical notion of
$\varepsilon$-containment in $2$-norm, introduced by Murray and von Neumann
\cite{MvNIV} and developed by Christensen \cite{Christensen79}.  The
conditional-expectation formulation in
\eqref{eq:near-inclusion-expectations} is used systematically by
Popa, Sinclair, and Smith \cite{PSS}.

The following observation will be used repeatedly.

\begin{lemma}\label{lem:uniform-near-inclusion}
If $A_n,B_n\subseteq M_{d_n}(\mathbb C)$ and
$\prod_{n\to\U}B_n\subseteq\prod_{n\to\U}A_n$, then there are
$\varepsilon_n\to_{\U}0$ such that
$B_n\subset_{2,\varepsilon_n}A_n$.
\end{lemma}

\begin{proof}
Otherwise, on a set belonging to $\U$, one can choose contractions
$b_n\in B_n$ with $\dist_2(b_n,A_n)$ bounded below.  The element
$(b_n)_{n\to\U}$ belongs to $\prod_{\U}B_n$ but not to
$\prod_{\U}A_n$, a contradiction.
\end{proof}

Replacing $d_n$ by $m_n$ with $m_n/d_n\to_{\U}1$, adding or deleting a
corner of normalized trace tending to zero, and conjugating by partial
isometries on the remaining corner do not change the tracial
ultraproduct.  We refer to such changes as \emph{stable negligible
modifications}.  They will always be accompanied by estimates showing
that the modified and original unit balls have $2$-Hausdorff distance
tending to zero.

Once such a modification has been made, we identify the old and new
ultraproducts and continue to use the symbols $d_n$, $\tr_n$, and
$\|\cdot\|_2$.  Thus harmless changes of normalized trace caused by a
dimension ratio tending to one will not be recorded repeatedly.

\begin{lemma}
\label{lem:dimension-invariance}
Let $m_n/d_n\to_{\U}1$.  For every sequence of unital
$*$-subalgebras $A_n\subseteq M_{m_n}(\mathbb C)$ there are unital
$*$-subalgebras $B_n\subseteq M_{d_n}(\mathbb C)$ such that, under the
canonical identification of the ambient tracial ultraproducts,
$\prod_{n\to\U}A_n=\prod_{n\to\U}B_n$.
Consequently, internality of a subalgebra is invariant under stable
negligible changes of matrix dimension.
\end{lemma}

\begin{proof}
If the dimensions remain bounded along $\U$, the assumption implies
$m_n=d_n$ on a set belonging to $\U$, and there is nothing to prove.
We may therefore assume that the dimensions tend to infinity along $\U$.
At indices where $m_n\leq d_n$, take
$B_n=A_n\oplus\mathbb C1_{d_n-m_n}$.
The added scalar corner has normalized trace tending to zero.

It remains to treat indices where $m_n>d_n$.  Suppress the index $n$,
put $k=m-d$, and write, after unitary conjugacy,
\[
 A=\bigoplus_{i=1}^t
   \bigl(M_{r_i}(\mathbb C)\otimes1_{s_i}\bigr)
   \subseteq M_m(\mathbb C),
 \qquad w_i=r_is_i.
\]
Choose $j$ such that $S:=\sum_{i<j}w_i\leq k<S+w_j$, and put
$q=k-S$, $h=\min\{r_j,s_j\}$, and
$q'=h\lceil q/h\rceil$.  Then
$0\leq e:=q'-q<h\leq\sqrt{w_j}\leq\sqrt m$.
Delete the first $j-1$ summands.  In the $j$-th summand delete $q'$
dimensions by replacing $r_j$ with $r_j-q'/s_j$ if $s_j\leq r_j$,
and by replacing $s_j$ with $s_j-q'/r_j$ if $r_j<s_j$.
Let $p$ be the support projection of the retained algebra and put
$C=pAp$, regarded as an algebra on $p\mathbb C^m$.  Then
\[
 \operatorname{rank}(p)=m-S-q'=d-e,
 \qquad \operatorname{rank}(1-p)=k+e=o_{\U}(m).
\]

Write $D_1$ for the operator-norm unit ball of an algebra $D$ and
$d_{H,2}^{(m)}$ for Hausdorff distance in the normalized $2$-norm of
$M_m(\mathbb C)$.  Compression sends $A_1$ into $C_1$ up to an error
supported on $1-p$, while every element of $C_1$ has a norm-preserving
block extension to $A_1$.  Hence
\[
 d_{H,2}^{(m)}(A_1,C_1)
 \leq 2\sqrt{\frac{k+e}{m}}=o_{\U}(1).
\]
Let $Q\in M_m(\mathbb C)$ be the standard projection of rank $d$ and
choose $R\leq Q$ of rank $d-e$.  Since both $p$ and $R$ have codimension
$k+e=o_{\U}(m)$, there is a unitary $u\in U(m)$ such that
$upu^*=R$ and $\|u-1\|_{2,m}=o_{\U}(1)$.  Set
\[
 \widetilde B=uCu^*\oplus\mathbb C(Q-R)
 \subseteq QM_m(\mathbb C)Q
 \cong M_d(\mathbb C),
\]
and let $B\subseteq M_d(\mathbb C)$ be the corresponding algebra.
The extra scalar corner has rank $e=o_{\U}(m)$, so the unit balls of
$A$ and $B$ have normalized $2$-Hausdorff distance tending to zero
under the standard corner identification.  Equality of the two
ultraproduct algebras follows.  Interchanging $d_n$ and $m_n$ gives the
final assertion.
\end{proof}

\section{Correcting near inclusions}
\label{sec:correction}

There is an extensive perturbation theory for finite von Neumann
subalgebras.  Besides the foundational work of
Christensen \cite{Christensen79}, the closest general results to what follows
are due to Popa, Sinclair, and Smith \cite{PSS}.  Under the symmetric
hypothesis $\|E_A-E_B\|_{\infty,2}\leq\delta$, they obtain, for small
$\delta$, large cutdowns of $A$ and $B$ that are spatially isomorphic through
a partial isometry close to the identity \cite[Theorem~5.2]{PSS}.  They also
prove that if $A\subseteq B$ is already an exact inclusion and
$B\subset_{2,\delta}A$, with $\delta<23^{-1/2}$, then $A$ and $B$ agree on a
corner cut out by a projection $p\in Z(A'\cap B)$ satisfying
$\tr(p)\geq1-23\delta^2$ \cite[Theorem~3.5]{PSS}.  On the other hand, their
one-sided example \cite[Proposition~5.5 and the preceding discussion]{PSS}
shows that, for general finite von Neumann algebras,
$B\subset_{2,\delta}A$ need not yield even a nonzero spatially embedded
corner of $B$ in $A$.

For a fixed finite-dimensional algebra, a related correction of approximate
matrix units in uniform $2$-norm is given by Carri\'on, Castillejos,
Evington, Gabe, Schafhauser, Tikuisis, and White
\cite[Proposition~8.8]{CCEGSTW}.  Its tolerance depends on the chosen
finite-dimensional source.  It therefore does not give a uniform statement
for a sequence $B_n$ whose number and sizes of matrix blocks may grow.

What is needed here is consequently a specifically finite-dimensional and
stable statement: the near inclusion is only one-sided, the algebras may
vary without a complexity bound, a corner of asymptotically negligible
dimension may be changed, and a common algebra $D_n$ must be carried to the
same corrected copy.  These additional features are the content of
Proposition~\ref{prop:bratteli-correction} below.  We include the proof for completeness.

Stable correction of approximate representations in normalized $2$-norm
was developed by Gowers and Hatami for finite groups \cite{GH} and by
de Chiffre, Ozawa, and the author for amenable groups \cite{dCOT}.
Compact-group and unitary-group forms appear in \cite{DBT,ATU}.
The relative finite-dimensional statement needed here follows from
a Stinespring dilation and a spectral cutoff.

\begin{proposition}
\label{prop:bratteli-correction}
Let $A_n,B_n\subseteq M_{d_n}(\mathbb C)$ be unital
finite-dimensional algebras and suppose
$B_n\subset_{2,\varepsilon_n}A_n$, where
$\varepsilon_n\to_{\U}0$.  After stable negligible modifications there
are exact inclusions $\widehat B_n\subseteq\widehat A_n$
whose two ultraproducts agree with those of $B_n$ and $A_n$,
respectively.  If $D_n\subseteq A_n\cap B_n$, the correction can be
chosen so that the two identifications carry $D_n$ to the same subalgebra
on asymptotically full corners.
\end{proposition}

\begin{proof}
It suffices to work on the $\U$-large set where $\varepsilon_n<1/2$.
We first suppress $n$ and use the unnormalized Hilbert--Schmidt norm.
The map
\[
 \Phi:B\otimes A'\longrightarrow M_d(\mathbb C),\qquad
 \Phi(b\otimes a)=E_A(b)a,
\]
is unital and completely positive: it is the composition of
$E_A|_B\otimes\id$ with the multiplication $*$-homomorphism
$A\otimes A'\to M_d(\mathbb C)$.  Take a finite-dimensional
Stinespring dilation \cite{Stinespring55}
$\Phi(x)=V^*\rho(x)V$, where $V:\mathbb C^d\to\mathcal K$ is an
isometry, and write $\rho_B,\rho_{A'}$ for the commuting factor
representations.  The restrictions of $\Phi$ to $D\otimes1$ and
$1\otimes A'$ are $*$-homomorphisms, for every $D\subseteq A\cap B$.
Consequently
\[
 \rho_B(d)V=Vd\quad(d\in D),\qquad
 \rho_{A'}(a)V=Va\quad(a\in A').
\]
For every $b\in B$, trace preservation and orthogonality of $E_A$ give
\[
 \|\rho_B(b)V-Vb\|_{\mathrm{HS}}^2
 =2\|b-E_A(b)\|_{\mathrm{HS}}^2.
\]

Put $P=VV^*$ and
$h=\int_{U(B)}\rho_B(b)P\rho_B(b)^*\,db$.
The operator $P$ commutes with $\rho_{A'}(A')$, so $h$ commutes with
both factor representations.  Averaging is an orthogonal projection
for the Hilbert--Schmidt inner product; hence
\[
 \Tr(h-h^2)=\|h-P\|_{\mathrm{HS}}^2
 =\int_{U(B)}\bigl(d-\|E_A(b)\|_{\mathrm{HS}}^2\bigr)\,db
 \leq\varepsilon^2d.
\]
Let $Q=1_{[1/2,1]}(h)$ and $m=\operatorname{rank}Q$.  Since
$\Tr h=d$ and $\Tr(QP)=\Tr(Qh)$, functional calculus gives
\[
 |m-d|\leq\|Q-P\|_{\mathrm{HS}}^2
 =\Tr\min\{h,1-h\}
 \leq2\varepsilon^2d.
\]
In particular $\|(1-Q)V\|_{\mathrm{HS}}^2\leq2\varepsilon^2d$.
Restrict the factor representations to $Q\mathcal K$, writing them as
$\pi_B,\pi_{A'}$, and put
\[
 \widehat B=\pi_B(B),\qquad \widehat A=\pi_{A'}(A')'.
\]
These algebras satisfy $\widehat B\subseteq\widehat A$.

Let $W$ be the polar part of $QV$.  Polar decomposition preserves
intertwining, so $W$ intertwines $A'$ and every common subalgebra $D$
exactly.  Singular-value comparison gives
\[
 \operatorname{rank}W\geq d-2\varepsilon^2d,
 \qquad \|W-V\|_{\mathrm{HS}}^2\leq4\varepsilon^2d.
\]
Thus the initial and final complements of $W$ have dimension
$O(\varepsilon^2d)$.  Restoring $n$, we have $m_n/d_n\to_{\U}1$;
with $\pi_{D,n}=\pi_{B,n}|_{D_n}$, the exact identities are
\begin{equation}
 \begin{aligned}
 W_n d&=\pi_{D,n}(d)W_n &&(d\in D_n),\\
 W_n a&=\pi_{A',n}(a)W_n &&(a\in A_n').
 \end{aligned}
 \label{eq:relative-intertwiners}
\end{equation}
The preceding estimates also give
\begin{equation}
 \sup_{b\in B_n,\,\|b\|\leq1}
 d_n^{-1/2}\|\pi_{B,n}(b)W_n-W_nb\|_{\mathrm{HS}}
 \leq(4+\sqrt2)\varepsilon_n\longrightarrow_{\U}0.
 \label{eq:uniform-B-intertwiner}
\end{equation}

We record the multiplicity estimates needed below.  Write
\[
 B_n=\bigoplus_i(M_{p_{n,i}}(\mathbb C)\otimes1_{q_{n,i}}),
 \qquad
 A_n=\bigoplus_j(M_{r_{n,j}}(\mathbb C)\otimes1_{s_{n,j}}).
\]
Since $\pi_{B,n}$ and $\pi_{A',n}$ are commuting representations of
these finite-dimensional algebras, there are integers $k_{n,ji}\geq0$
such that
\[
 Q_n\mathcal K_n=\bigoplus_{j,i}
 (\mathbb C^{p_{n,i}}\otimes\mathbb C^{s_{n,j}})^{\oplus k_{n,ji}}.
\]
Put $r'_{n,j}=\sum_i k_{n,ji}p_{n,i}$ and
$q'_{n,i}=\sum_j k_{n,ji}s_{n,j}$.  The exact $B_n$-intertwiner
\[
 T_n=\int_{U(B_n)}\pi_{B,n}(b)Q_nV_nb^*\,db
\]
satisfies $\|T_n-V_n\|_{\mathrm{HS}}\leq
2\sqrt2\varepsilon_n\sqrt{d_n}$, and therefore
$\operatorname{rank}T_n\geq d_n-8\varepsilon_n^2d_n$.
Comparing multiplicities with $T_n$ and $W_n$, respectively, gives
\[
 \sum_i p_{n,i}\min\{q_{n,i},q'_{n,i}\}
 \geq\operatorname{rank}T_n,\qquad
 \sum_j s_{n,j}\min\{r_{n,j},r'_{n,j}\}
 \geq\operatorname{rank}W_n.
\]
Both original total dimensions equal $d_n$, and both corrected totals
equal $m_n=d_n+O(\varepsilon_n^2d_n)$.  It follows that
\begin{equation}
 \frac1{d_n}\sum_j s_{n,j}|r'_{n,j}-r_{n,j}|
 +\frac1{d_n}\sum_i p_{n,i}|q'_{n,i}-q_{n,i}|
 \longrightarrow_{\U}0.
 \label{eq:marginal-errors}
\end{equation}

Finally, use $W_n$ to identify the two ambient spaces after adding
negligible complementary spaces and completing it to a unitary $U_n$.
Extend algebras on added corners by scalars.  Equation
\eqref{eq:uniform-B-intertwiner} and surjectivity of
$\pi_{B,n}:B_n\to\widehat B_n$, including on unit balls, give both
Hausdorff estimates for $B_n$ and $\widehat B_n$ under this
identification.  For $A_n$, let $p_n,q_n$ be the initial and final
supports of $W_n$.  Equation~\eqref{eq:relative-intertwiners} gives
$p_n\in A_n$, $q_n\in\widehat A_n$, and, by taking commutants on
the support corners,
\[
 W_n(p_nA_np_n)W_n^*=q_n\widehat A_nq_n.
\]
Compression to these supports changes contractions by $o_2(1)$
uniformly, and contractions in the corner algebras extend by zero.
Thus the same identification gives
\begin{equation}
 d_{H,2}\bigl((B_n)_1,(\widehat B_n)_1\bigr)+
 d_{H,2}\bigl((A_n)_1,(\widehat A_n)_1\bigr)
 \longrightarrow_{\U}0.
 \label{eq:algebra-hausdorff}
\end{equation}
Here $d_{H,2}$ is the Hausdorff distance of operator-norm unit balls
for the normalized $2$-norm.  This proves equality of the respective
ultraproducts.  The exact $D_n$-equivariance of $W_n$ supplies the
relative assertion on its asymptotically full supports; no equivariance
of the unitary completion is needed.
\end{proof}

\section{The no-drift argument}
\label{sec:no-drift}

The conditional median relative to $D_n$ turns the order between the
corrected scale operators into equality in the ultraproduct.  The
following finite-dimensional estimate then gives the reverse near
inclusion directly.  Remark~\ref{rem:drift-example} explains why the
common internal subalgebra is essential.

\begin{lemma}
\label{lem:reverse-scale-estimate}
Let $B\subseteq C\subseteq M_d(\mathbb C)$ be unital subalgebras, with
block decompositions
\[
 B=\bigoplus_i(M_{p_i}(\mathbb C)\otimes1_{q_i}),
 \qquad C=\bigoplus_j(M_{r_j}(\mathbb C)\otimes1_{s_j}).
\]
Let $w_i,z_j$ be their minimal central projections and put
\[
 \Delta_B=\sum_i\frac{p_i}{q_i}w_i,\qquad
 \Delta_C=\sum_j\frac{r_j}{s_j}z_j,\qquad
 R=\Delta_C\Delta_B^{-1}.
\]
Then $R\geq1$ and, for the trace-preserving conditional expectation
$E_B\colon C\to B$,
\begin{equation}
 \sup_{a\in C,\,\|a\|\leq1}\|a-E_B(a)\|_2
 \leq\sqrt{2}\,\tr(1-R^{-1})^{1/2}.
 \label{eq:reverse-scale-estimate}
\end{equation}
\end{lemma}

\begin{proof}
Let $(k_{ji})$ be the Bratteli matrix, so that
$r_j=\sum_i k_{ji}p_i$ and $q_i=\sum_jk_{ji}s_j$.
On the nonzero cell $z_jw_i$, the value of $R$ is
$r_jq_i/(s_jp_i)\geq k_{ji}^2\geq1$.

We first compute the mean squared distance to $B$ over $U(C)$.
Write the $j$-th block of $u\in U(C)$ using occurrences
$(i,\alpha)$, $1\leq\alpha\leq k_{ji}$.  Then
\[
 E_B(u)_i=\frac1{q_i}\sum_{j,\alpha}s_j
                   u_{j,(i,\alpha),(i,\alpha)}.
\]
The blocks $u_j$ are independent Haar unitaries in $U(r_j)$, and
their entries satisfy
$\int (u_j)_{ab}\overline{(u_j)_{cd}}\,du_j
=\delta_{ac}\delta_{bd}/r_j$.  Consequently,
\[
 \int_{U(C)}\|E_B(u)\|_2^2\,du
 =\frac1d\sum_{j,i}\frac{k_{ji}p_i^2s_j^2}{q_ir_j}
 =\tr(R^{-1}).
\]
Since $E_B$ is an orthogonal projection in $L^2$, this gives
\[
 \delta:=\int_{U(C)}\|u-E_B(u)\|_2^2\,du
 =\tr(1-R^{-1}).
\]

For the uniform bound, take commutants in $M_d(\mathbb C)$ and also
write $E_B$ for the ambient conditional expectation.  Conditional
expectations are Haar averages over the appropriate commutant unitary
groups.  Thus, for every $v\in U(C)$,
\begin{align*}
 \|v-E_B(v)\|_2^2
 &=\frac12\int_{U(B')}\|[v,w]\|_2^2\,dw\\
 &\leq2\int_{U(B')}\|w-E_{C'}(w)\|_2^2\,dw
 =2\delta.
\end{align*}
Here $E_{C'}(w)$ commutes with $v$, and the last equality follows
from Fubini: both integrals defining that equality are one half of
$\int_{U(B')}\int_{U(C)}\|[u,w]\|_2^2\,du\,dw$.
Finally, every contraction in $C$ is a convex combination of unitaries,
so convexity proves \eqref{eq:reverse-scale-estimate}.
\end{proof}

\begin{theorem}
\label{thm:no-drift}
Let $A=\prod_{n\to\U}A_n$ and $D=\prod_{n\to\U}D_n$ be internal
finite-dimensional subalgebras of a tracial matrix
ultraproduct, with $D_n\subseteq A_n$.  Let
$u_1,\ldots,u_m\in U(\M)$ have representatives
$u_{\ell,n}\in D_n'$.  Assume $u_\ell^*Au_\ell\subseteq A$ for
$1\leq\ell\leq m$ and
$D=A\cap\bigcap_{\ell=1}^m\{u_\ell\}'$.  Then
$u_\ell^*Au_\ell=A$ for $1\leq\ell\leq m$.
\end{theorem}

\begin{proof}
Write
$A_n=\bigoplus_{i\in I_n}(M_{r_{n,i}}(\mathbb C)\otimes1_{s_{n,i}})$,
with minimal central projections $z_{n,i}$, and put
\[
 \Delta_n=\sum_i\frac{r_{n,i}}{s_{n,i}}z_{n,i}\in Z(A_n)_+.
\]
For commuting positive invertible operators, write
$\beta_S(T)=T(T+S)^{-1}$.  If $e$ is a minimal central projection of
$D_n$, then
\[
 E_{D_n}(e\beta_t(\Delta_n))=c_{n,e}(t)e.
\]
Indeed, this expectation belongs to $eD_n$ and commutes with $D_n$.
The scalar $c_{n,e}(t)$ is continuous and strictly decreasing from $1$
to $0$.  Choose $m_{n,e}>0$ with $c_{n,e}(m_{n,e})=1/2$, and set
\[
 m_n=\sum_e m_{n,e}e\in Z(D_n)_+,
 \qquad x_n=\beta_{m_n}(\Delta_n).
\]
Thus $E_{D_n}(x_n)=\frac12 1$.  We call $m_n$ the
$D_n$-conditional median of $\Delta_n$.

Fix $\ell$, and put $A_n^+=A_n$,
$A_n^-=u_{\ell,n}^*A_nu_{\ell,n}$, and
$y_n=u_{\ell,n}^*x_nu_{\ell,n}$.
Lemma~\ref{lem:uniform-near-inclusion} gives
$A_n^-\subset_{2,\varepsilon_n}A_n^+$ with
$\varepsilon_n\to_{\U}0$.  Since $u_{\ell,n}\in D_n'$, both
algebras contain $D_n$.  Apply the relative form of
Proposition~\ref{prop:bratteli-correction}, obtaining
$\widehat A_n^-\subseteq\widehat A_n^+$ on a space of dimension
$\widehat d_n$.

Suppress $n$ from the block indices.  Label the source blocks using
conjugation by $u_{\ell,n}$, so that their original sizes and
multiplicities are again $r_i,s_i$.  The Bratteli matrix $(k_{ji})$ of
the corrected inclusion gives
\[
 \widehat A_n^-=
 \bigoplus_{i:s'_i>0}(M_{r_i}(\mathbb C)\otimes1_{s'_i}),\qquad
 \widehat A_n^+=
 \bigoplus_{j:r'_j>0}(M_{r'_j}(\mathbb C)\otimes1_{s_j}),
\]
where $r'_j=\sum_i k_{ji}r_i$ and $s'_i=\sum_jk_{ji}s_j$.
Equation~\eqref{eq:marginal-errors} gives
\begin{equation}
 \frac1{d_n}\sum_j s_j|r'_j-r_j|
 +\frac1{d_n}\sum_i r_i|s'_i-s_i|\longrightarrow_{\U}0.
 \label{eq:special-marginal-errors}
\end{equation}
In particular, $\widehat d_n/d_n\to_{\U}1$.
Write $\widehat w_i,\widehat z_j$ for the corrected central
projections and set
\[
 \Delta_n^-=\sum_{i:s'_i>0}\frac{r_i}{s'_i}\widehat w_i,
 \qquad
 \Delta_n^+=\sum_{j:r'_j>0}\frac{r'_j}{s_j}\widehat z_j.
\]
On every nonzero cell $\widehat z_j\widehat w_i$,
\begin{equation}
 R_n:=\Delta_n^+(\Delta_n^-)^{-1}
 =\frac{r'_js'_i}{r_is_j}\geq k_{ji}^2\geq1.
 \label{eq:scale-order-cell}
\end{equation}
Here the scalar equality describes the restriction of $R_n$ to the
cell.  Thus $\Delta_n^-\leq\Delta_n^+$.

Let $\pi_{D,n}$ be the common representation of $D_n$ furnished by
the correction, and put $\widehat m_n=\pi_{D,n}(m_n)$ and
$x_n^\pm=\beta_{\widehat m_n}(\Delta_n^\pm)$.
The three operators $\Delta_n^-,\Delta_n^+,\widehat m_n$ commute,
since the first is central in $\widehat A_n^-$, the second is central
in $\widehat A_n^+$, and the third belongs to $\widehat A_n^-$.
Consequently,
\begin{equation}
 0<x_n^-\leq x_n^+<1.
 \label{eq:bounded-scale-order}
\end{equation}

We verify that these corrected bounded scales represent the original
ones.  Let $W_n$ be the common partial isometry furnished by
Proposition~\ref{prop:bratteli-correction}.  Define the raw transported
scales on the corrected space by
\[
 \begin{aligned}
 \widetilde x_n^-=\pi_{B,n}(y_n)
 =\beta_{\widehat m_n}
       \left(\sum_{i:s'_i>0}\frac{r_i}{s_i}\widehat w_i\right), \quad
 \widetilde x_n^+=\beta_{\widehat m_n}
       \left(\sum_{j:r'_j>0}\frac{r_j}{s_j}\widehat z_j\right).
 \end{aligned}
\]
Here $B_n=A_n^-$, and the source formula uses
$u_{\ell,n}m_n=m_nu_{\ell,n}$.  Since $y_n$ is a contraction of
$A_n^-$, the uniform estimate \eqref{eq:uniform-B-intertwiner} gives
\[
 \|\widetilde x_n^-W_n-W_ny_n\|_2\longrightarrow_{\U}0.
\]
On the target side, $W_n$ exactly intertwines $D_n$ and $(A_n^+)'$.
In particular it intertwines $m_n$ and $\Delta_n$, whose central
projections have the coordinate standard labels $j$.  Functional
calculus on its retained supports therefore gives
\[
 \widetilde x_n^+W_n=W_nx_n.
\]

The elementary scalar bound
\begin{equation}
 \left|\frac{a}{a+s}-\frac{b}{b+s}\right|^2
 \leq\left|\frac ab-1\right|\qquad(a,b,s>0)
 \label{eq:uniform-scale-bound}
\end{equation}
is uniform in $s$: the difference is at most both $1$ and
$|a/b-1|$, since it equals
$|a/b-1|bs/((a+s)(b+s))$.
Apply this bound on the full corrected blocks, with
$a/b=s_i/s'_i$ on source block $i$ and $a/b=r_j/r'_j$ on target
block $j$.  Their dimensions are $r_is'_i$ and $r'_js_j$,
respectively, so
\[
 \begin{aligned}
 \|x_n^--\widetilde x_n^-\|_{2,\widehat d_n}^2
 &\leq\frac1{\widehat d_n}\sum_i r_i|s'_i-s_i|,\\
 \|x_n^+-\widetilde x_n^+\|_{2,\widehat d_n}^2
 &\leq\frac1{\widehat d_n}\sum_j s_j|r'_j-r_j|.
 \end{aligned}
\]
Both bounds tend to zero by \eqref{eq:special-marginal-errors}.
No uniform bound on $m_n$ or $m_n^{-1}$ is needed.

Use the common unitary completion $U_n$ of $W_n$ from the correction,
after stabilization to dimension
$d_n^\circ=\max\{d_n,\widehat d_n\}$, and extend the bounded scales
by zero on the added summands.  Denote the resulting norm by
$\|\cdot\|_{2,\circ}$.  The complementary corners have dimension
$o_{\U}(d_n)$, so the preceding estimates imply
\begin{equation}
 \|x_n^-U_n-U_ny_n\|_{2,\circ}
 +\|x_n^+U_n-U_nx_n\|_{2,\circ}
 \longrightarrow_{\U}0.
 \label{eq:scale-transport}
\end{equation}

Put $x=(x_n)_{n\to\U}$.  Under this common stable identification,
\eqref{eq:scale-transport} identifies $(x_n^-)_{n\to\U}$ with
$u_\ell^*xu_\ell$ and $(x_n^+)_{n\to\U}$ with $x$.
Thus \eqref{eq:bounded-scale-order} passes to
$u_\ell^*xu_\ell\leq x$ in $\M$.  The two elements have equal trace,
so faithfulness of the trace gives $u_\ell^*xu_\ell=x$.
This holds for every $\ell$, whence
$x\in A\cap\bigcap_\ell\{u_\ell\}'=D$.
Coordinatewise conditional expectation now yields
\[
 x=E_D(x)=(E_{D_n}(x_n))_{n\to\U}=\tfrac12 1.
\]
Using \eqref{eq:scale-transport} once more, we obtain
$\|x_n^\pm-\frac12 1\|_{2,\widehat d_n}\to_{\U}0$ on the corrected
space, since the stabilizing corners have vanishing relative dimension.

It remains to apply Lemma~\ref{lem:reverse-scale-estimate}.
Functional calculus for the three commuting operators gives
\[
 R_n^{-1}
 =\frac{x_n^-(1-x_n^+)}{x_n^+(1-x_n^-)}.
\]
The joint spectral measures of $(x_n^-,x_n^+)$ concentrate at
$(1/2,1/2)$, where the scalar function on the right is continuous
and equals $1$.  Thus $R_n^{-1}\to1$ in measure.  Since
$0\leq R_n^{-1}\leq1$, bounded convergence gives
\[
 \widehat{\tr}_n(1-R_n^{-1})\longrightarrow_{\U}0.
\]
Lemma~\ref{lem:reverse-scale-estimate} therefore places every
contraction of $\widehat A_n^+$ within $o_{\U}(1)$ in normalized
$2$-norm of the contraction
$E_{\widehat A_n^-}(a)\in\widehat A_n^-$, uniformly in $a$.
The common identification \eqref{eq:algebra-hausdorff} gives the
reverse near inclusion $A_n^+\subset_{2,o_{\U}(1)}A_n^-$.
Passing to the ultraproduct gives $A\subseteq u_\ell^*Au_\ell$.
Together with the assumed inclusion, this proves the theorem.
\end{proof}

\begin{remark}
\label{rem:drift-example}
The following standard example shows why a finite-level anchor is
needed.  Let
\[
 A_n=\bigoplus_{i=1}^n M_{2^i}(\mathbb C),
 \qquad
 \tau_n(x_1,\ldots,x_n)
   =\frac1n\sum_{i=1}^n\tr_{2^i}(x_i),
\]
so that the trace is distributed evenly among the $n$ simple summands.
This trace is the restriction of the normalized matrix trace under the
realization
\[
 A_n=\bigoplus_{i=1}^n
       \bigl(M_{2^i}(\mathbb C)\otimes1_{2^{n-i}}\bigr)
       \subseteq M_{n2^n}(\mathbb C).
\]
Define
\[
 \theta_n(x_1,\ldots,x_n)
   =(0,x_1\otimes 1_2,x_2\otimes 1_2,\ldots,
      x_{n-1}\otimes 1_2).
\]
Thus $\theta_n$ shifts every summand one place to the right, doubles
its defining representation, and annihilates the final summand
$M_{2^n}(\mathbb C)$.  Each $\theta_n$ is an exact nonunital,
noninjective $*$-homomorphism.  For every contraction
$x=(x_1,\ldots,x_n)$,
\[
 \|1-\theta_n(1)\|_{2,\tau_n}^2=\frac1n,
 \qquad
 |\tau_n(\theta_n(x))-\tau_n(x)|\leq\frac1n,
\]
and
\[
 0\leq \|x\|_{2,\tau_n}^2-\|\theta_n(x)\|_{2,\tau_n}^2
   =\frac1n\|x_n\|_{2,\tr_{2^n}}^2\leq\frac1n.
\]
Consequently $(\theta_n)_n$ induces a unital trace-preserving injective
$*$-endomorphism $\theta$ of $\prod_{n\to\U}(A_n,\tau_n)$.
This endomorphism is implemented by a unitary in the ambient tracial
ultraproduct.  Indeed, let $w_n\in U(n2^n)$ cyclically permute the $n$
summand spaces of dimension $2^n$, sending the $i$-th space to the
$(i+1)$-st and the last space to the first.  Choose the identifications so
that conjugation by $w_n$ agrees with $\theta_n$ on the first $n-1$ input
summands.  For every contraction $x=(x_1,\ldots,x_n)\in A_n$,
\[
 \|w_nxw_n^*-\theta_n(x)\|_2^2
 =\frac1n\|x_n\|_{2,\tr_{2^n}}^2\leq\frac1n.
\]
Hence $w=(w_n)_{n\to\U}$ implements $\theta$ in the ambient ultraproduct:
$\theta(a)=waw^*$ for every $a\in\prod_{n\to\U}(A_n,\tau_n)$.
It is not surjective.  Indeed, set
$v=\operatorname{diag}(1,-1)\in M_2(\mathbb C)$ and
\[
 y_n=(v,1_2\otimes v,1_4\otimes v,\ldots,
          1_{2^{n-1}}\otimes v)\in U(A_n).
\]
In the first summand the range of $\theta_n$ is zero, while in the
$j$-th summand, for $j\geq2$, it is
$M_{2^{j-1}}(\mathbb C)\otimes1_2$.  The trace-preserving conditional
expectation of $1_{2^{j-1}}\otimes v$ onto this subalgebra is zero.
Therefore $\dist_{2,\tau_n}(y_n,\theta_n(A_n))=1$ for every $n$, and
$(y_n)_{n\to\U}$ does not lie in the range of $\theta$.

The lost input and output boundary summands each have trace $1/n$, so
their defects disappear in the ultraproduct, while the proper matrix
amplification persists in every summand.  Thus asymptotic unitality and
trace preservation alone do not prevent drift.  It is likewise not
enough merely to observe that each finite center has finitely many
atoms.  The internality of $D$, together with the identity
$E_{D_n}(x_n)=\frac12 1$, provides the finite-level anchor that excludes
this phenomenon and is essential to the argument.  In the permutation
proof of \cite{KT}, the corresponding role is played by the ambient
Kazhdan component decomposition and its weighted median.
\end{remark}

\begin{remark}
\label{rem:baumslag-solitar-drift}
The distinction between genuine finite quotients and general sofic
approximations is already visible in
$BS(1,2)=\langle a,t\mid tat^{-1}=a^2\rangle$.  In a height
approximation, level $i$ may be taken to consist of $2^{n-i}$
$a$-cycles of length $2^i$, so every level has the same size.  Away from
the two end levels, $t$ sends each cycle onto one of the two
$a^2$-orbits in a cycle at the next level.  This is the permutation
version of the shift in Remark~\ref{rem:drift-example}.  In contrast,
in the finite quotients
$C_{2^n-1}\rtimes C_n$, the generator of $C_n$ acts by multiplication
by $2$, so $\langle a^2\rangle=\langle a\rangle$ and there is no drift.

This behavior of finite quotients is not a consequence of the
abelianity of $\langle a\rangle$.  If $q\colon G\to Q$ is a homomorphism
to a finite group and $tHt^{-1}\leq H$, then
$q(t)q(H)q(t)^{-1}\leq q(H)$.  The two groups have the same order, so
they are equal.  Thus a compressor always normalizes the finite image
of $H$.
\end{remark}

\section{Proof of the main results}
\label{sec:main-results}

\subsection{From the centralizer problem to normalization}

We now assume a positive answer to the centralizer problem and prove
Theorem~\ref{thm:normalization-intro}.

\begin{proposition}
\label{prop:coherent-models}
Let $H<G$ be Kazhdan groups and let
$\pi\colon G\to U(\M)$ be a homomorphism.  Put
$A=\pi(H)'\cap\M$ and $D=\pi(G)'\cap\M$.
After stable negligible modifications, there are finite-dimensional
algebras $D_n\subseteq A_n$ such that
$A=\prod_{n\to\U}A_n$ and $D=\prod_{n\to\U}D_n$.
If $t_1,\ldots,t_m\in G$ and $u_\ell=\pi(t_\ell)$, representatives may
be chosen so that $u_{\ell,n}\in D_n'$ for all $n$ and $\ell$.
\end{proposition}

\begin{proof}
Apply the centralizer problem to $\pi|_H$ and $\pi|_G$.  Since both
groups are Kazhdan, there are finite-dimensional algebras
$D_n,A'_n\subseteq M_{d_n}(\mathbb C)$ such that
$A=\prod_{n\to\U}A'_n$ and $D=\prod_{n\to\U}D_n$.

Since $D\subseteq A$,
Lemma~\ref{lem:uniform-near-inclusion} gives
$D_n\subset_{2,\varepsilon_n}A'_n$ with
$\varepsilon_n\to_{\U}0$.  Apply
Proposition~\ref{prop:bratteli-correction} to this near inclusion and
use the common stable identification supplied by that proposition.
Denote the corrected algebras by $D_n\subseteq A_n$; they still
represent $D\subseteq A$.

Now choose unitary representatives $v_{\ell,n}$ of $u_\ell$ in the
corrected ambient dimensions.  Since $D\subseteq\{u_\ell\}'$, a
diagonal argument gives
$\sup_{d\in U(D_n)}\|[v_{\ell,n},d]\|_2\to_{\U}0$.
Indeed, otherwise one could choose $d_n\in U(D_n)$ for which the
commutators stay bounded away from zero; then $(d_n)_{n\to\U}\in D$
would not commute with $u_\ell$.  Haar averaging $v_{\ell,n}$ under
conjugation by $U(D_n)$ gives an operator in $D_n'$ at
$o_2(1)$-distance from $v_{\ell,n}$.  Its polar part and support
projections also lie in $D_n'$.  Complete the polar part blockwise on
its initial and final kernels to a unitary $u_{\ell,n}\in D_n'$.
Since $v_{\ell,n}$ is unitary, the polar-decomposition estimate gives
$\|u_{\ell,n}-v_{\ell,n}\|_2\to_{\U}0$.  Thus the new sequence still
represents $u_\ell$.
\end{proof}

\begin{proof}[Proof of Theorem~\ref{thm:normalization-intro}]
The group $G$ is finitely generated, and $P_H$ generates $G$.
Consequently there are $t_1,\ldots,t_m\in P_H$ such that
$G=\langle H,t_1,\ldots,t_m\rangle$.  Put
$A=\pi(H)'\cap\M$ and $D=\pi(G)'\cap\M$.
Apply Proposition~\ref{prop:coherent-models} and choose representatives
$u_{\ell,n}\in D_n'$ of $u_\ell=\pi(t_\ell)$.
The choice of generators gives
$D=A\cap\bigcap_{\ell=1}^m\{u_\ell\}'$.  Thus
Proposition~\ref{prop:coherent-models} supplies all the hypotheses
of Theorem~\ref{thm:no-drift} except the one-sided inclusions.

Let $x\in A$ and $h\in H$.  Put $k=t_\ell ht_\ell^{-1}\in H$.  Then
$[u_\ell^*xu_\ell,\pi(h)]=u_\ell^*[x,\pi(k)]u_\ell=0$.
Thus $u_\ell^*Au_\ell\subseteq A$.  By Theorem~\ref{thm:no-drift}, equality
holds for every $\ell$.  The algebra $A$ is fixed pointwise by
$\pi(H)$ and normalized by all $u_\ell$, hence it is normalized by
$\pi(G)$.
\end{proof}

\subsection{A conditional nonhyperlinear group}

We finish the proof of Theorem~\ref{thm:main-intro}.  

\begin{proof}[Proof of Theorem~\ref{thm:main-intro}]

Because $H$ is infranormal but not normal, there is a strict compressor
$t\in P_H$, that is, $tHt^{-1}<H$.
Indeed, if every element of $P_H$ normalized $H$, then the group
generated by $P_H$ would normalize $H$.

Put $P=G*_H G$.  Write $G_1$ and $G_2$ for its two factors, and write $g_i$ for
the copy of $g\in G$ in $G_i$.  Suppose that $P$ is hyperlinear and
choose an injective homomorphism $\pi\colon P\to U(\M)$.  Put
$A=\pi(H)'\cap\M$.  The element $c=t_1^{-1}t_2$ centralizes $H$:
for $h\in H$, the element $t_2ht_2^{-1}$ belongs to the common subgroup
$H$ and hence
$t_1^{-1}(t_2ht_2^{-1})t_1=h$.  Thus $\pi(c)\in A$.

Apply Theorem~\ref{thm:normalization-intro} to the first copy $G_1$.
It follows that $A$ is normalized by $\pi(G_1)$, and therefore
$\pi(t_2t_1^{-1})=\pi(t_1ct_1^{-1})\in A$.
Choose $h\in H\setminus tHt^{-1}$.  The commutator
$[h,t_2t_1^{-1}]$ is nontrivial: after collecting letters from the same
factor, its normal form is
$
 (ht_2)(t_1^{-1}h^{-1}t_1)t_2^{-1}.
$
A strict compressor cannot belong to $H$, so the first and third letters
lie outside $H$; the middle letter lies outside $H$ precisely because
$h\notin tHt^{-1}$.  The word is therefore reduced.  On the other hand,
$\pi(t_2t_1^{-1})\in A$ commutes with $\pi(h)$, so this nontrivial
commutator lies in the kernel of $\pi$, a contradiction.

Kun and the author proved in \cite[Theorem~E]{KT} that, for every prime power
$q$ and $r,d\geq3$, the group
$H=\EL_r(\mathbb F_q[x_1,\ldots,x_d])$ and the group
\[
 G=\EL_r(\mathbb F_q[x_1^{\pm1},\ldots,x_d^{\pm1}])
       \rtimes\operatorname{SL}_d(\mathbb Z)
\]
are residually finite Kazhdan groups and that $H$ is infranormal but not
normal in $G$.  
In fact, one may see a strict compressor explicitly.  If
$u_{12}=I+E_{12}\in\operatorname{SL}_d(\mathbb Z)$ and
$t=(1,u_{12})\in G$, then
$tHt^{-1}=\EL_r(\mathbb F_q[x_1,x_1x_2,x_3,\ldots,x_d])<H$.
For example, an elementary matrix with coefficient $x_2$ lies in $H$
but not in $tHt^{-1}$.
\end{proof}

\section*{Acknowledgments}

A first version of this paper circulated beginning of August 2026 and
created with the help of GPT 5.6 as an interactive proof assistant. The
proof was further simplified with the help of OpenAI's GPT Astra, which
was also used to assist in drafting and editing parts of
this manuscript.  All content was reviewed and substantially revised by
the author, who is responsible for
the final text.

\end{document}